\documentclass[12pt]{elsart}
\usepackage[english]{babel}
\usepackage{amsmath,amssymb,enumerate,xcolor}

\newenvironment{proof}{ {\it Proof.} }{\hfill{\it{QED}}\medskip}
\catcode`\<=\active \def<{
\fontencoding{T1}\selectfont\symbol{60}\fontencoding{\encodingdefault}}
\catcode`\>=\active \def>{
\fontencoding{T1}\selectfont\symbol{62}\fontencoding{\encodingdefault}}
\catcode`\|=\active \def|{
\fontencoding{T1}\selectfont\symbol{124}\fontencoding{\encodingdefault}}
\newcommand{\assign}{:=}
\newcommand{\mathd}{\mathrm{d}}
\newcommand{\nocomma}{}

\newcommand{\tmop}[1]{\ensuremath{\operatorname{#1}}}

\newcommand{\nobracket}{}
\newcommand{\tmdummy}{$\mbox{}$}
\newcommand{\br}{{\bf r}}
\newcommand{\bs}{{\bf s}}
\newcommand{\bX}{{\bf X}}
\newcommand{\bY}{{\bf Y}}
\newcommand{\bx}{{\bf x}}
\newcommand{\by}{{\bf y}}

\newcommand{\cD}{\mathcal{D}}
\newcommand{\w}{\omega}
\newcommand{\N}{{\mathbb N}}
\newcommand{\R}{\mathbb{R}}
\newcommand{\C}{\mathbb{C}}

\newtheorem{theorem}{Theorem}[section]
\newtheorem{lemma}[theorem]{Lemma}

\newtheorem{remark}[theorem]{Remark}

\begin{document}

\begin{frontmatter}
  \title{A linear dimensionless bound for the weighted Riesz vector}
  
  \author[IMT]{Komla Domelevo}  \ead{komla.domelevo@math.univ-toulouse.fr}
  \author[IMT]{Stefanie Petermichl\thanksref{SP}\thanksref{fn1}}  \ead{stefanie.petermichl@math.univ-toulouse.fr}
  \author[WC]{Janine Wittwer} \ead{jwittwer@westminstercollege.edu}

  \thanks[SP]{Research supported in part by ANR-12-BS01-0013-02. The author is a member of IUF}
  \thanks[fn1]{Correponding author, Tel:+33 5 61 55 76 59, Fax: +33 5 61 55 83 85}
  \address[IMT]{Universit\'e Paul Sabatier, \\Institut de Math\'ematiques de Toulouse, \\118 route de Narbonne, \\F-31062 Toulouse, France}
  \address[WC]{Westminster College, \\1840 South 1300 East, \\Salt Lake City, UT 84105, USA}
  
  \begin{abstract}
We show that the norm of the vector of Riesz transforms as operator in the
weighted Lebesgue space $L^2_{\w}$  is
bounded by a constant multiple of the first power of the Poisson-$A_2$ characteristic of $\w$. The bound is free of dimension and optimal. Our argument requires an extension of Wittwer's linear estimate for martingale transforms to the vector valued setting with scalar weights, for which we indicate a proof. Extensions to $L^p_{\w}$ for $1<p<\infty$ are discussed.   Our arguments to exhibit sharpness at the critical exponent $p=2$ require a martingale extrapolation theorem, for which we provide a proof. We also show that for $n>1$, the Poisson-$A_2$ class is properly included in the classical $A_2$ class.       
\end{abstract}


\end{frontmatter}

\section{Introduction}
A weight is a positive $L^1_{loc}$ function. Muckenhoupt proved in \cite{Mu} that for $1<p< \infty$ the maximal function is bounded on $L^p_{\w}$ iff the weight $\w$ belongs to the class $A_p$, where 
\[ \w \in A_p \text{ iff } Q_p(\w) := \sup_{B}\langle \w \rangle_B \langle \w^{-1/(p-1)}\rangle_B^{p-1} \, < \infty.\]
Here the notation $\langle \cdot \rangle_B$ denotes the average over the ball $B$ and the supremum runs over all balls $B$. Hunt, Muckenhoupt, Wheeden proved in \cite{HMW} that the $A_p$ condition also characterizes the boundedness of the Hilbert transform 
\[Hf(x)= \frac1{\pi}\int \frac{f(y)}{x-y}dy\]
in $L^p_{\w}$. The extension of this theory to general Calder\'{o}n-Zygmund operators was done by Coifman and Fefferman in \cite{CF}.

One often sees the restriction to $p=2$ when working with weights. It stems from the availability of a theory of extrapolation initiated by Rubio de Francia \cite{RdF}. 

Quantitative norm estimates for these operators in dependence on $Q_p(\w)$ or $Q_2(\w)$ in particular, have attracted considerable interest. The linear and optimal bound in terms of the classical $A_2$ characteristic $Q_2(\w)$ has been established for the Hilbert transform by one of the authors in \cite{P2} and then in \cite{P3} for the higher dimensional case and Riesz transforms. In \cite{DGPP} it has been observed that linear, sharp estimates in the case $p=2$ for operators such as Riesz transforms or Haar multipliers extrapolate via Rubio de Francia's theorem to optimal constants for other $p$. The same upper estimates hold for all 
Calder\'{o}n-Zygmund operators, which was first shown in \cite{Hy}. See also \cite{HPTV}. This remarkable result has been reproven  by Lerner \cite{Ler2013a} using a completely different approach. The bound depends upon the dimension in all these proofs. 

The focus in this note is on the Riesz vector in weighted spaces $L^2_{\w}$ and the norm dependence on dimension as well as quantities related to $Q_2(\w)$. We are interested in a version of $A_2$ (and $A_p$) which is particularly well-suited for working with the Riesz transforms in $\R^n$, where we exploit the intimate connection of Riesz transforms and harmonic functions. Namely, we use the Poisson-$A_2$ class with characteristic $\tilde{Q}_2(\w)$, which considers Poisson averages instead of box averages in the definition of $A_2$. This allows us to obtain a bound free of dimension for the Riesz vector $\vec{R}$:
$$\|\vec{R}\|_{L^2_{\w}\to L^2_{\w}}\lesssim \tilde{Q}_2(\w).$$

The Poisson-$A_2$ characteristic arises naturally from the viewpoint of martingales driven by space-time Brownian motion as in Gundy-Varopoulos \cite{GV}: the Riesz transforms of a function can be written as conditional expectation of a simple transformation of a martingale associated to the function. The Poisson-$A_2$ class is adapted to the stochastic process considered (see \cite{IK}) and for this reason, we can find a simple, structural proof that is independent of dimension. 

Our argument is deterministic, using transference principle through Bellman functions, where convexity is replaced by harmonicity. This was also the approach in \cite{NT} as well as \cite{PW} for the weighted Hilbert transform and \cite{VD} for unweighted Riesz transforms. 


Interestingly,  one-dimensional Poisson extensions of weights made a reappearance in the works concerned with the famous two-weight problem for the Hilbert transform, see \cite{Lac2013a}, \cite{LacSawSheUri2012a} and \cite{Hy2014}. It enjoys its interpretation as a `tamed' Hilbert transform, a feature that appears to be lost in higher dimensions. In the one-dimensional case, we see a quadratic relation between the Poisson characteristic and the classical characteristic, but the classes themselves are the same. Interestingly, these different $A_2$ classes are not identical when the dimension is larger. We will show examples of $A_2$ weights whose Poisson integral diverges when the dimension is at least two. Such weights belong to $A_2$ but not to Poisson-$A_2$. This shows that the Poisson characteristic used on a pair of weights such as for the two-weight problems, is not necessary in higher dimension. This is one of several obstacles when considering the two-weight question for the Riesz transforms, that is currently under investigation. We mention the recent advance  \cite{LacWic2013a} where the Poisson characteristic is modified. 

To see sharpness, the Buckley examples can be used for any $1<p<2$. This is in a contrast to the classical case, where they work directly for $1<p\le 2$ and one then uses duality to reach the remaining $p$. So, in the Poisson case, there remains a gap at $p=2$. Previous texts claiming otherwise use a reference that contains an arithmetic error. The authors are indebted to the anonymous referee for finding this mistake. It has lead to an interesting detour to show sharpness at the exponent 2, through a martingale representation of the Hilbert transform and a martingale extrapolation theorem, whose proof (very similar to the known proof for sublinear operator) we sketch.

\section{Notation}

The Riesz transforms $R_k$ in $\R^n$ are the component operators  of the Riesz vector $\vec{R}$, defined on the Schwartz class by 
\[\hat{(R_k f)}  (\xi) = i\frac{\xi_k}{\|\xi\|}\hat f (\xi).\]
We consider the space $L^2_{\w}$, where $\w$ is a positive scalar valued $L^1_{\rm loc}$ function, 
called a weight. More specifically, the space $L^2_{\w}(\R^n;\C)$ consists of all measurable functions $f:\R^n \to \C$ so that the quantity
\[\|f\|_{\w} := \left(\int_{\R^n}|f(x)|^2 \w(x) dx\right)^{1/2}\] is finite, where $dx$ denotes the Lebesgue measure on $\R^n$. For the space of vector valued functions $L^2_{\w}(\R^n;\C^n)$, we replace 
$|\cdot |$ by the $\ell^2$ norm $\|\cdot \|$.

We are concerned with a special class of weights, called Poisson-$A_2$. We say $\w \in \tilde{A}_2$  if 
\begin{equation}\label{apdef}
\tilde{Q}_2(\w) := \sup_{(x,t)\in \R^n\times \R_+} P_t(\w)(x)P_t(\w^{-1})(x) <\infty
\end{equation} 
where $P_t$ denotes the Poisson extension operator into the upper half space defined by 
\[P_t = e^{-tA}\] where we define $A:=\sqrt{-\Delta}$ and where $\Delta$ is the Laplacian in $\R^n$. The scalar Riesz transforms can be written as \[R_k=\partial_k \circ A^{-1}.\]
The Poisson kernel has the form \[P_t(y)=c_n\frac{t}{(t^2+|y|^2)^{\frac{n+1}2}}\] where $c_n$ is its normalizing factor. The extension operator becomes \[P_tf(x)=\int_{\R^n}f(y)P_t(x-y)dy.\]

\section{Main results}\label{mainres}

The main purpose of this text is to provide the dimensionless estimate:
\begin{theorem} \label{lbd}
There exists a constant $c$ that does not depend on the dimension $n$ or on the weight $\w$ so that for all weights $\w \in \tilde{A}_2$ the Riesz vector  as an operator in weighted space $L^2_{\w} \to L^2_{\w}$ has operator norm $\| \vec{R}\|_{L^2_{\w} \to L^2_{\w}} \le c \tilde{Q}_2(\w).$
\end{theorem}
The estimate is sharp in the following sense: there exists no function $\Phi:[1,\infty[\to \mathbb{R}_+$ so that $\frac{\Phi(x)}{x}\to 0$ when $x\to \infty$ with $\| \vec{R}\|_{L_{\w}^2\to L_{\w}^2}\le c\Phi(\tilde{Q}_2(\w))$ for all weights $\w\in \tilde{A}_2$ and all $n$.

A similar estimate holds for other exponents $1<p<\infty$ and is optimal as well. It can be found in section \ref{section_p}. 

We also investigate the relationship between different Muckenhoupt classes.  Notably, their relation changes with dimension:
\begin{theorem} \label{PoissonA2neqA2}
Poisson-$A_2$ and classical $A_2$ only define the same classes of weights when the dimension is one:  $\tilde{A}_2=A_2$ if and only if  $n=1$. Otherwise $\tilde{A}_2$ is properly included in $A_2$.
\end{theorem}

In our method of proof, the dimensionless estimate in Theorem \ref{lbd} requires us to prove a vector valued version of a theorem by Wittwer (see section \ref{section_p} for  $p\neq 2$.)

\begin{theorem}\label{Wittwervector}
Let $\mathcal{H}$ be a separable Hilbert space and $\vec{f}:\mathbb{R} \to \mathcal{H}$. For each $I$ in the dyadic collection $\mathcal{D}$, let $\sigma_I$ denote any unitary transformation on  $\mathcal{H}$ and $h_I$ a Haar function. The vector valued operator 
$$\vec{T}_{\sigma}f=\sum_{I\in \mathcal{D}}\sigma_I(\vec{f},h_I)h_I$$
has operator norm uniformly bounded by $cQ_2(\w)$.
\end{theorem}

In section \ref{section_extrapolation} we also prove a sharp extrapolation theorem in the martingale setting for filtered spaces with continuous index.

\section{The dimension-free estimate}\label{pf}

Since 
\[\|\vec{R}\|_{L^2_{\w}(\R^n;\C)\to L^2_{\w}(\R^n;\C^n)}=\|\w^{1/2}\vec{R}\w^{-1/2}\|_{L^2(\R^n;\C)\to L^2(\R^n;\C^n)}\] where the outer multiplication by $\w^{1/2}$ is a scalar multiplication. We can estimate $\|\vec{R}\|_{L^2_{\w}\to L^2_{\w}}$  
via $L^2$ duality. It is sufficient to estimate 
\[|(\vec{g},\w^{1/2}\vec{R}\w^{-1/2}f)|\le c\tilde{Q}_2(\w)\|f\|\|\vec{g}\|\]
for test functions (smooth and compactly supported) $f, \vec{g}$, where $f$ is scalar valued and $\vec{g}$ vector valued. Or (considering $\w^{-1/2}f$ instead of $f$ and $\w^{1/2}\vec{g}$ instead of $\vec{g}$):
\[|(\vec{g},\vec{R}f)|\le c \tilde{Q}_2(\w) \|\vec{g}\|_{\w^{-1}}\|f\|_{\w}.\]

To prove this estimate, we prove the following theorem:
\begin{theorem}\label{bydu}
For test functions $f,\vec{g}$ and $\w \in \tilde{A}_2$ we have the following estimate:
\begin{equation}
|(\vec{g},\vec{R} f)| \le c\tilde{Q}_2(\w)(\|\vec{g}\|_{\w^{-1}}^2+\|f\|_{\w}^2);
\end{equation}
here $c$ does not depend on $f,\vec{g},n,k$ or $\w$.
\end{theorem}
Considering $\lambda f$ and $\lambda ^{-1}\vec{g}$ for appropriate $\lambda$, with the considerations above  yields Theorem \ref{lbd}.

Before we turn to the proof of Theorem \ref{bydu}, let us formulate several useful lemmata.

\subsection{Several useful Lemmata}

The following is a well known fact. It is, for example, stated in \cite{GV}.
\begin{lemma}\label{grf}
\[(\vec{g},\vec{R}f)=-4\int_0^{\infty}(\frac{d}{dt}P_t\vec{g},\nabla P_t f)tdt.\]
\end{lemma}
The proof using semigroups is very simple and concise, so we include it for the convenience of the reader. Instead of using semigroups, the same result can be obtained by the use of the Fourier transform.

\begin{proof}
Observe that $F(0)=\int_0^{\infty}F''(t)tdt$ for sufficiently fast decaying $F$. 
So \[(g,R_kf)=(P_0g,P_0R_kf)=\int_0^{\infty}\frac{d^2}{dt^2}(P_tg,P_tR_kf)\;\;tdt.\]
The right hand side becomes
\[\int_0^{\infty}\left((\frac{d^2}{dt^2}P_tg,P_tR_kf)+2(\frac{d}{dt}P_tg,\frac{d}{dt}P_tR_kf)+(P_tg,\frac{d^2}{dt^2}P_tR_kf)\right) tdt.\]
Now we use the fact that $\frac{d}{dt}P_t=-AP_t$ and $\frac{d^2}{dt^2}P_t=A^2P_t$ and symmetry of $A$ to see that the above equals \[4\int_0^{\infty}(AP_tg,AP_tR_kf)tdt.\] 
Observing that $A$ commutes with $P_t$ and $\partial_k$, that $R_k=\partial_k \circ A^{-1}$, and using $\frac{d}{dt}P_t=-AP_t$, we obtain
\begin{equation*}
(g,R_kf)=-4\int_0^{\infty}(\frac{d}{dt}P_tg,\partial_kP_tf)tdt.
\end{equation*}
For function $f$ and vector function $\vec{g}$  this becomes
\begin{equation*}
(\vec{g},\vec{R} f)=-4\int_0^{\infty}(\frac{d}{dt}P_t\vec{g},\nabla P_tf)tdt.
\end{equation*}
\end{proof}

Our final estimate is based on a sharp weighted estimate for a dyadic model operator in one dimension that we now describe. Let $\mathcal{D}=\{2^k[n;n+1):n,k\in \mathbb{Z}\}$ denote the standard dyadic grid in $\R$. Let for $I\in \mathcal{D}$ denote $I_{\pm}\in \mathcal{D}$ the respective right and left halves of the interval $I$. Then, $h_I=|I|^{-1/2}(\chi_{I_+}-\chi_{I_-})$ form the Haar basis normalized in $L^2$. Let $\sigma$ denote a sequence $\sigma_I=\pm 1$. By $T_{\sigma}$ we mean $$T_{\sigma}f=\sum_{I\in \mathcal{D}}\sigma_I(f,h_I)h_I.$$ Wittwer's estimate from \cite{W}, is $$\sup_{\sigma}\|T_{\sigma}\|_{L^2_{\w}\to L^2_{\w}}\le c Q_2(\w)$$ with $c$ independent of the weight. This estimate allows one to claim the existence of a Bellman function such as in \cite{PV}, however only for the case of real-valued functions. In the vector-valued case, we will need Theorem \ref{Wittwervector}, which in turn implies the existence of a Bellman function adapted to vector-valued functions arising in our problem:

\begin{lemma}\label{L:Bellman}
For any $Q>1$ let $\cD$ be a subset of $\R\times \R \times \C \times \C^n \times \R\times\R$
$$\cD_Q = \{(\bX,\bY,\bx,\by,\br,\bs) : |\bx|^2<\bX\bs, \,\|\by\|^2<\bY\br, \,1<\br\bs<Q\}.$$ 
For any compact $K \subset \cD_Q$ there exists an infinitely differentiable function $B_{K,Q}$ defined in a small neighborhood of $K$ that still lies inside $\cD_Q$
so that the following estimates hold in $K$.
\begin{equation}\label{size}0\le B_{K,Q}\le c Q(\bX+\bY),\end{equation}
\begin{equation}\label{concavity}-d^2B_{K,Q}\ge 2 \; |d\bx| \; \|d\by\|.\end{equation}
 \end{lemma}
The last inequality describes an operator inequality where the left hand side is the negative Hessian of $B$. Notice that one of the variables, namely $\by$, is vector-valued.

\begin{proof}(of Theorem \ref{Wittwervector}) To prove this estimate, first observe that the estimate for all choices of $\sigma_I$ is proved by duality if we show the estimate
\[ \frac{1}{| J |} \sum_{I \in \mathcal{D} (J)} \| (\vec{f}, h_I) \| \|
   (\vec{g}, h_I) \| \leqslant c Q_2 (\w) \langle  \| \vec{f} \|^2 \w \rangle^{1
   / 2}_J \langle  \| \vec{g} \|^2 \w^{- 1} \rangle^{1 / 2}_J. \]
   One splits the estimate into the classical four sums.
  Let for a weight $\nu$ denote $h^{\nu}_I$ the disbalanced Haar
functions, forming an orthonormal basis in $L^2(\nu)$. Here $h_I = \alpha^{\nu}_I h^{\nu}_I + \rho^{\nu}_I \chi_I | I |^{-1/2}$, where one
calculates $$\rho^{\nu}_I = \frac{\langle \nu \rangle_{I_+}- \langle \nu \rangle_{I_-}}{\langle \nu \rangle_I}  {\text{ and }}
\alpha^{\nu}_I = \frac{\langle \nu \rangle^{1 / 2}_{I_+}\langle \nu \rangle^{1 / 2}_{I_-}}{\langle \nu \rangle^{1 / 2}_{I}}$$ and so obtains four
sums
\[ \mathbf{I}= \frac{1}{| J |} \sum_{I \in \mathcal{D} (J)} | \alpha^{w^{-
   1}}_I | | \alpha^{\w}_I | \| (\vec{f} \w, h^{\w^{ - 1}}_I)_{\w^{- 1}} \| \| (\vec{g}
   \w^{- 1}, h^{\w}_I)_{\w} \| \]
\[ \mathbf{I}\mathbf{I}= \frac{1}{| J |} \sum_{I \in \mathcal{D} (J)} |
   \alpha^{\w^{- 1}}_I |  | \rho^{\w}_I | \| (\vec{f} \w, h^{\w^{- 1}}_I)_{\w^{- 1}}
   \| \| (\vec{g}, \chi_I | I |^{-1/2}) \| \]
\[ \mathbf{I}\mathbf{I}\mathbf{I}= \frac{1}{| J |} \sum_{I \in \mathcal{D}
   (J)} | \alpha^{\w}_I |  | \rho^{\w^{- 1}}_I | \| (\vec{f}, \chi_I| I |^{-1/2}) \|
   \| (\vec{g} \w^{- 1}, h^{\w}_I)_{\w} \| \]
\[ \mathbf{I}\mathbf{V}= \frac{1}{| J |} \sum_{I \in \mathcal{D} (J)} |
   \rho^{\w^{}}_I |  | \rho^{\w^{- 1}}_I | \| (\vec{f}, \chi_I| I |^{-1/2})\| \|
   (\vec{g}, \chi_I| I |^{-1/2}) \|. \]
Sum $\mathbf{I} \lesssim Q_2 (\w)^{1 / 2} \langle  \| \vec{f} \|^2 \w \rangle^{1
/ 2}_J \langle  \| \vec{g} \|^2 \w^{- 1} \rangle^{1 / 2}_J$ is estimated via
Cauchy Schwarz and by using that we have an orthonormal basis in the weighted
spaces, and  $| \alpha^{\w^{- 1}}_I | | \alpha^{\w}_I | \lesssim Q_2 (\w)^{1 /
2}$. Sums $\mathbf{I}\mathbf{I}$ and $\mathbf{I}\mathbf{I}\mathbf{I}$ are
similar. We estimate $\mathbf{I}\mathbf{I} \lesssim \langle  \| \vec{f} \|^2 \w
\rangle^{1 / 2}_J \left( \frac{1}{| J |} \sum_{I \in \mathcal{D} (J)} |
\alpha^{\w^{- 1}}_I |^2  | \rho^{\w}_I |^2 \| (\vec{g}, \chi_I| I |^{-1/2}) \|^2 \right)^{1 /
2} .$ The second part involves use of a Carleson embedding theorem for vector
functions: 
\begin{equation*}
\frac{1}{| K |} \sum_{J \in \mathcal{D} (K)} a_J \langle \w
\rangle^2_J \lesssim \langle \w \rangle_K \; \forall K \Rightarrow \frac{1}{| J
|} \sum_{I \in \mathcal{D} (J)} a_I \| (\vec{g}, \chi_I| I |^{-1}) \|^2 \lesssim \langle
\|\vec{g} \w^{- 1 / 2}\|^2 \rangle_J
\end{equation*} 
applied to $a_I =| I |  | \alpha^{w^{- 1}}_I |^2  | \rho^{\w}_I|^2$. The testing condition on the left hand side is a scalar estimate and was shown in \cite{W}. To see the proof of the implication above, one 
estimates $\| (\vec{g}, \chi_I| I |^{-1}) \|\le \langle \|\vec{g}\| \rangle_I$ and uses the scalar weighted Carleson embedding theorem, see \cite{NTV} p. 911. 
Sum $\mathbf{I}\mathbf{V}$ can be handled in the same way as the short cut in \cite{RTV} p. 7 using the maximal function after estimating $\| (\vec{g}, \chi_I| I |^{-1}) \|\le \left( \|\vec{g}\w^{-1}\|\w,\chi_I | I |^{-1}\right)$ and similar for $f$. 
%
\end{proof}

\begin{proof}(of Lemma \ref{L:Bellman}) Using Theorem \ref{Wittwervector}, by duality and localising, one obtains an inequality of the form
$$
\frac1{|J|}\sum_{I\in \mathcal{D}(J)} \|(\vec{f},h_I) \|\|(\vec{g},h_I) \|\le cQ_2(\w)\langle\|\vec{f}\|^2\w\rangle^{1/2}_J  \langle\|\vec{g}\|^2\w^{-1}\rangle^{1/2}_J.
$$
By setting up an extremal problem in the same way as in \cite{PV} p. 294 (scalar weigted) or \cite{PSW} p. 320 (vector, unweighted) one obtains the existence of a Bellman function with variables $\bX=\langle \|\vec{f}\|^2\w \rangle, \bY=\langle \|\vec{g}\|^2\w^{-1}\rangle, \bx=\langle \vec{f}\rangle, \by=\langle \vec{g}\rangle, \br=\langle \w\rangle, \bs=\langle \w^{-1} \rangle$. Replacing in \cite{PV} p. 294 $x^2$ by $\|\bx\|^2$ and $y^2$ by $\|\by\|^2$ in the description of the domain, one obtains the Bellman function as
claimed in Lemma \ref{L:Bellman}, where the scalar parameter $\bx$ stands for the vector parameter $(\bx,0,\ldots,0)$.
\end{proof}

The estimate (\ref{concavity}) on the Hessian is not quite enough for us.
We will need the following form of a Lemma that has been proven in \cite{VD} and generalised in \cite{DTV}, the so-called `ellipse lemma'.
\begin{lemma}\label{ellipse}
Let $m,l,k \in \N$. Denote $d=m+l+k$. For arbitrary $u\in \R^{m+l+k}$ write $u=u_m \oplus u_l \oplus u_k$, where $u_i \in \R^i$ for $i=m,l,k$. Let $r_m=\|u_m\|, r_l=\|u_l\|$. Suppose the matrix $A\in \R^{d\times d}$ is such that
\[(Au,u)\ge 2r_m r_l\] for all $u\in \R^d$. Then there exists $\tau >0$ so that 
\[(Au,u)\ge\tau r_m^2+\tau^{-1}r_l^2\]
for all $u\in \R^d$.
\end{lemma}
We will be using this lemma for $m=2$, $l=2n$ and $k=4$.

\subsection{Proof of the dimension-free estimate}

Recall the inequality of Theorem \ref{bydu}. We want to show for test functions $f$ and $g$ and $\w \in \tilde{A}_2$: 
$$|(\vec{g},\vec{R}f)| \le c\tilde{Q}_2(\w)(\|\vec{g}\|_{\w^{-1}}^2+\|f\|_{\w}^2).$$
For a fixed non-constant weight $\w$ we let $Q=(1+\varepsilon)\tilde{Q}_2(\w)$. This gives rise to the set $\cD_Q$. We define 
\[b_{K,Q}(x,t) = B_{K,Q}(v(x,t))\]
where
\[v(x,t)=\left(P_t(|f|^2\w),P_t(\|\vec{g}\|^2\w^{-1}),P_t(f),P_t(\vec{g}),P_t(\w),P_t(\w^{-1})\right)(x)\]
Here $K$ is a compact subset of $\cD_Q$ to be chosen later. 

Note that the vector $v \in \cD_Q$ for any choice of $(x,t)$. This is ensured by $Q=\tilde{Q}_2(\w)$ and several applications of Jensen's inequality. Notice also that the vector $v$ takes compacts inside the interior of $\R^{n+1}_+$ to compacts $K$ inside $\cD_Q$ for fixed $f,\vec{g},\w$. 
By elementary application of the chain rule (using harmonicity of the components of $v$) one shows that
\[\Delta_{x,t}b(x,t)=\sum_{i=1}^n(d^2B(v)\frac{\partial}{\partial x_i}v,\frac{\partial}{\partial x_i}v)+(d^2B(v)\frac{\partial}{\partial t}v,\frac{\partial}{\partial t}v).\]
Here $\Delta_{x,t}$ is the full Laplacian in the upper half space 
\[\Delta_{x,t}=\sum_{i=1}^n \partial^2_{x_i}+\partial^2_t.\]
Notice that condition (\ref{concavity}) in Lemma \ref{L:Bellman} means that at any $v=(\bX,\bY,\bx,\by,\br,\bs)$ in $K\subset \cD_Q$, for any $u=(u_1,\ldots,u_6)\in \R\times \R \times \C \times \C^n \times \R\times\R$, we have the inequality 
\[(-d^2B_{K,Q}(v)u,u)\ge 2 \; |u_3| \; \|u_4\|.\]  
In our situation, $f,\vec{g},\w$ and $Q$ are fixed, but we have varying $K,x,t$. So 
Lemma \ref{ellipse} guarantees the existence of $\tau_{x,t,K}$ so that 
\[(-d^2B(v)\frac{\partial}{\partial x_i}v,\frac{\partial}{\partial x_i}v)
\ge 
\tau_{x,t,K}\lvert \frac{\partial}{\partial x_i}P_tf\rvert ^2+
\tau^{-1}_{x,t,K}\| \frac{\partial}{\partial x_i}P_t\vec{g}\|^2\]
for all $i$ and 
\[(-d^2B(v)\frac{\partial}{\partial t}v,\frac{\partial}{\partial t}v)
\ge \tau_{x,t,K}\lvert \frac{\partial}{\partial t}P_tf\rvert ^2+
\tau^{-1}_{x,t,K}\| \frac{\partial}{\partial t}P_t\vec{g}\|^2.\]
So 
\begin{eqnarray*}
\lefteqn{-\Delta_{x,t}b_{K,Q}(x,t)}\\
&\ge& 
\tau_{x,t,K}
\left(\sum_{i=1}^n \lvert \frac{\partial}{\partial x_i}P_tf\rvert ^2 + \lvert \frac{\partial}{\partial t}P_tf\rvert ^2\right)\\
&&+ 
\tau^{-1}_{x,t,K}
\left(\sum_{i=1}^n \| \frac{\partial}{\partial x_i}P_t\vec{g}\| ^2 + \| \frac{\partial}{\partial t}P_t\vec{g}\| ^2\right)\\
&\ge&
2\left(\sum_{i=1}^n \lvert \frac{\partial}{\partial x_i}P_tf\rvert ^2 + \lvert \frac{\partial}{\partial t}P_tf\rvert ^2\right)^{1/2}
\left(\sum_{i=1}^n \|\frac{\partial}{\partial x_i}P_t\vec{g}\| ^2 + \| \frac{\partial}{\partial t}P_t\vec{g}\| ^2\right)^{1/2}\\
&\ge&
2\left(\sum_{i=1}^n \lvert \frac{\partial}{\partial x_i}P_tf\rvert ^2\right)^{1/2} \| \frac{\partial}{\partial t}P_t\vec{g}\|\\ 
&=&2\|  \nabla P_t f \|  \| \frac{\partial}{\partial t}P_t\vec{g}\| 
\label{lapest}
\end{eqnarray*}
Using Lemma \ref{grf}, and the estimate for the Laplacian we just proved, we have:
\begin{eqnarray*}
\lefteqn{|(\vec{g},\vec{R} f)|}\\[.2em]
&\le&
4\int_0^{\infty}
\lvert (\frac{\partial}{\partial t}P_t\vec{g}, \nabla P_tf)\rvert tdt\\[.2em]
&\le&
4\int_0^{\infty}\int_{\R^n}
\|\frac{\partial}{\partial t}P_t\vec{g}\| \|  \nabla P_tf\| dxtdt\\[.2em]
&\le&
2\int_0^{\infty}\int_{\R^n}
-\Delta_{x,t}b_{K,Q}(x,t)dx tdt.
\end{eqnarray*}
It remains to see that
\begin{equation}
  \label{remains} - \int_{0}^{\infty} \int_{{\R}^{n}} \Delta_{x,t}
  b_{K,Q} (x,t) dxtdt \le C \tilde{Q}_{2} ( \|f\|_{\omega}^{2}
  +\|\vec{g}\|_{\omega^{-1}}^{2} ) 
\end{equation}
with $C$ independent of $n$. In order to obtain this last estimate, we will apply Green's formula as well
as some properties of our Bellman function. We are going to pass through
values of the function $b$.

Recall the statement of Green's formula:
\begin{theorem}
  \[ \int_{\Omega} \left(f (x) \Delta g (x) -g (x) \Delta f (x) \right)dA (x) =
     \int_{\partial \Omega} \left( f(t) \frac{\partial g}{\partial n} (t)-g(t)
     \frac{\partial f}{\partial n} (t) \right) dS (t) \]
  where $n$ is the outward normal and $dS$ the surface measure on $\partial
  \Omega$.
\end{theorem}
In order to be accurate, we are obliged to take care of a few technicalities
first.

Let $T_{R}$ be a cylinder with square base in upper half space $[-R,R]^{n}
\times [0,2R]$. For $R$ not too small, the point $(0,1)$ lies inside $T_{R}$.
Let $T_{R, \epsilon} =T_{R} + (0, \epsilon )$. For any interior point $( \xi ,
\tau )$, \ let $G^{R, \epsilon} [ (x,t),( \xi , \tau ) ]$ be its Green's
function, meaning that
\[ \Delta_{x,t} G^{R, \epsilon} [ (x,t),( \xi , \tau ) ] =- \delta_{( \xi ,
   \tau )}  \quad \text{\tmop{and}  } \quad G^{R, \epsilon} =0 \quad \text{\tmop{on}  } \quad \partial
   T_{R, \epsilon} . \]
Notice that $RT_{1,0} =T_{R, \epsilon} - (0, \epsilon )$ and the Green's
functions relate as follows:
\begin{lemma}
  {\tmdummy}
  The Green's function has the following scaling property:
  \begin{equation}
    R^{n-1} G^{R, \epsilon} [ (x,t),( \xi , \tau ) ] =G^{1,0} [ ( R^{-1} (x,t-
    \epsilon ),R^{-1} ( \xi , \tau - \epsilon ) ) ] .\label{scaling}
  \end{equation}
\end{lemma}
\begin{proof}
  By uniqueness it suffices to see that $R^{- (n-1)} G^{1,0}  [ R^{-1} (x,t-
  \epsilon ),R^{-1} ( \xi , \tau - \epsilon ) ]$ is indeed the Green function
  for the region $T_{R, \epsilon}$ at the point $( \xi , \tau )$. It is clear
  that it equals zero on $\partial T_{R, \epsilon}$. Furthermore for any test
  function $f$ we have
  \begin{eqnarray*}
    \lefteqn{ \int \int_{T_{R, \epsilon}} \Delta_{x,t} R^{- ( n-1 )} G^{1,0} [
    R^{-1} ( x,t- \epsilon ) ,R^{-1} ( \xi , \tau - \epsilon ) ] f ( x,t )
    dx dt}\\
    &  & \hspace{1em} =  \int \int_{T_{1,0}} \Delta_{y,s} G^{1,0} [ ( y,s )
    ,R^{-1} ( \xi , \tau - \epsilon ) ] f ( R y,R s+ \varepsilon )  dy
     ds\\
    &  & \hspace{1em} =-f ( \xi , \tau )
  \end{eqnarray*}
  We did a substitution $(x,t) = (Ry,Rs+ \epsilon )$. Note that there is a
  $R^{-2}$ factor arising from the switch of $\Delta_{x,t}$ to $\Delta_{y,s}$
  and a $R^{n+1}$ factor arising from the determinant.
\end{proof}

Recall that the vector $v$ maps each $T_{R, \epsilon}$ into a compact $K=K_{R,
\epsilon} \subset \mathcal{D}_{Q}$. For technical reasons we have to exhaust
the upper half space by compacts denoted by $M$. For that, first fix any compact set $M$ in the open upper
half space and consider $R$ large enough and $\epsilon$ small enough so that
$M \subset T_{R, \epsilon}$.

Let us start to use the size estimate of our Bellman function to obtain an
estimate of the function value $b_{K,Q}  (0,R+ \epsilon )$ from above:
\begin{eqnarray*}
  \lefteqn{b_{K,  Q} ( 0,R+ \epsilon ) }\\
   &\leq&  C  \tilde{Q}_{2} \left( P_{R+ \epsilon}( |f|^{2} \omega) ( 0 ) +P_{R+ \epsilon} (\|\vec{g}\|^{2} \omega^{-1} )( 0 ) \right)\\
  & = & c_{n} C  \tilde{Q}_{2} \int_{{\R}^{n}} |f|^{2} ( y ) \omega
  ( y ) \frac{R+ \epsilon}{((R+ \epsilon )^{2} +|y|^{2} )^{\frac{n+1}{2}}} dy
  \\
  &  & + c_{n} C  \tilde{Q}_{2}\int_{{\R}^{n}} \|\vec{g}\|^{2} (y) \omega^{-1} (y) \frac{R+ \epsilon}{((R+\epsilon )^{2} +|y|^{2} )^{\frac{n+1}{2}}} dy \nonumber\\
  & \leq & c_{n} ( R+ \epsilon )^{-n} C  \tilde{Q}_{2} ( \| f
  \|^{2}_{\omega} + \| \vec{g} \|_{\omega^{-1}}^{2} ) .  \label{bfromabove}
\end{eqnarray*}
For an estimate from below, Green's formula applied to our situation gives:
\begin{eqnarray*}
  \lefteqn{b_{K,Q} ( 0,R+ \epsilon ) }\\
  & = & - \int \int_{T_{R, \epsilon}} G^{R,
  \epsilon} ((x,t),(0,R+ \epsilon )) \Delta_{x,t} b_{K,Q} (x,t) dxdt\\
  &  & \hspace{1em} - \int_{\partial T_{R, \epsilon}} b_{K,Q} (x,t)
  \frac{\partial G^{R, \epsilon} ((x,t),(0,R+ \epsilon ))}{\partial n} dxdt \\
  &  & \hspace{1em} + \int_{\partial T_{R, \epsilon}} G^{R, \epsilon}
  ((x,t),(0,R+ \epsilon )) \frac{\partial b_{K,Q} ((x,t))}{\partial n} dxdt 
\end{eqnarray*}
The first boundary term is negative because $b$ is non-negative
and the outward normal of the Green's function is negative on the boundary of
$T_{R, \epsilon}$. The second boundary term vanishes because $G^{R, \epsilon} =0$
on the boundary. So we have the following estimate:
\begin{eqnarray*}
  \lefteqn{b_{K,Q}  (0,R+ \epsilon ) }\\
  & \geq & - \int \int_{T_{R, \epsilon}} G^{R,
  \epsilon} ((x,t),(0,R+ \epsilon )) \Delta_{x,t} b_{K,Q} (x,t) dxdt.\\
  & \geq & - \int \int_{M} G^{R, \epsilon} ((x,t),(0,R+ \epsilon ))
  \Delta_{x,t} b_{K,Q} (x,t) dxdt.
\end{eqnarray*}
since $- \Delta b \geq 0$ and where we recall that $M \subset T_{R,
\epsilon}$. We continue the estimate using the scaling properties of the Green
functions (\ref{scaling}).
\begin{eqnarray*}
  b_{K,Q}  (0,R+ \epsilon ) & \geq & - \int \int_{M} R^{- (n-1)} G^{1,0} 
  ( R^{-1} (x,t- \epsilon ),(0,1) ) \Delta_{x,t} b (x,t) dxdt.
\end{eqnarray*}
Since $G^{1,0} ( (R^{-1} x,0),(0,1) ) =0$ we have
\begin{eqnarray*}
  \lefteqn{b_{K,Q}  (0,R+ \epsilon ) }\\& \geq & - \int \int_{M} R^{- (n-1)}  \left\{
  G^{1,0}  ( ( R^{-1} x,R^{-1} (t- \epsilon ),(0,1) ) -
  \nobracket \right. \\
  &  &  \hspace{3em}\left. G^{1,0 } ( (R^{-1} x,0),(0,1) )  \right\}
  \Delta_{x,t} b (x,t) dxdt \\
  & = & - \int \int_{M} R^{- (n-1)}  \frac{\partial G^{1,0}}{\partial t} 
  (R^{-1} x, \tau ) R^{-1}  (t- \epsilon ) \Delta_{x,t} b_{K,Q} (x,t) dxdt
  \\
  & = & - \int \int_{M} R^{-n}  \frac{\partial G^{1,0}}{\partial t}  (R^{-1}
  x, \tau ) \Delta_{x,t} b_{K,Q} (x,t) dx (t- \epsilon ) dt  ,
\end{eqnarray*}
where $0 \le \tau \le R^{-1}  (t- \epsilon )$. Pulling this all together with
the estimate from above,
\begin{eqnarray*}
  \lefteqn{- \int \int_{M} R^{-n} \frac{\partial G^{1,0}}{\partial t} ( R^{-1} x, \tau
  ) \Delta_{x,t} b_{K,Q} ( x,t ) d x ( t- \varepsilon ) d t }\\ 
  &\leq & b_{K,Q}  (0,R+ \epsilon )
   \leq  c_{n}  (R+ \epsilon )^{-n} C \tilde{Q}_{2}  (
  \|f\|_{\omega}^{2} +\|\vec{g}\|_{\omega^{-1}}^{2} ) ,
\end{eqnarray*}
hence
\[ - \int \int_{M} \frac{\partial G^{1,0}}{\partial t}  (R^{-1} x, \tau )
   \Delta_{x,t} b_{K,Q} (x,t) dx (t- \epsilon ) dt \leq c_{n} C \tilde{Q}_{2} 
   ( \|f\|_{\omega}^{2} +\|\vec{g}\|_{\omega^{-1}}^{2} ) \]
uniformly with respect to $R$ and $\epsilon$, for all given $M$. When $R$ goes
to infinity, the normal derivative $\frac{\partial G^{1,0}}{\partial t} 
(R^{-1} x, \tau )$ tends to $\frac{\partial G^{1,0}}{\partial t}  (0,0)$
uniformly with respect to $( x,t ) \in M$. But we know that the normal
derivative $\frac{\partial G^{1,0}}{\partial t}  (0,0)$ is exactly the
normalizing factor $c_{n}$ of the Poisson kernel. Letting $R$ go to infinity and
$\varepsilon$ go to zero yields for all compact $M$ of the upper half space:
\[ - \int \int_{M} \Delta_{x,t} b_{K,Q} (x,t) dxtdt \leq C \tilde{Q}_{2}  (
   \|f\|_{\omega}^{2} +\|\vec{g}\|_{\omega^{-1}}^{2} ) . \]
Finally, letting $M$ exhaust the upper half space establishes (\ref{remains}).
This concludes the proof of Theorem \ref{bydu} and therefore the proof
of the main Theorem \ref{lbd}.

\section{The comparison of classical and Poisson characteristic.}

In this section we prove Theorem \ref{PoissonA2neqA2}. We provide an example that demonstrates that $\tilde{A}_2 \neq A_2$ if $n>1$.
For the case $n=1$, it is known that the two classes are the same.  In
fact, for $n=1$  the estimates 
$${Q}_2(\w)\lesssim \tilde{Q}_2(\w) \lesssim {Q}_2(\w)^2$$ are
proven in \cite{H} and this estimate is sharp. For the lower estimate the cork screw point is used (see also below) and for the upper estimate one splits the arising Poisson integrals into dyadic rings and uses the doubling property of the $A_2$ weight repeatedly.  

If $n>1$ however, an easy example shows that the
Poisson integral of a simple power weight diverges, although the
weight belongs to classical $A_2$. Consider
$\w_{\alpha}(x)=|x|^{\alpha}$.  It is well known and straightforward
to check that $\w_{\alpha} \in A_2$ if and only if $|\alpha|<n$. Also
$Q_2(\w_{\alpha})\sim \frac1{n^2-\alpha^2}$. We show that the Poisson
integral $P_t\w_{\alpha}(0)$ diverges for $\alpha>1$. Indeed,
\begin{eqnarray*}
\lefteqn{P_t(\w_{\alpha})(0) }\\
&\sim &\int_{\R^n}  \frac{t}{(t^2+|x|^2)^{\frac{n+1}{2}} }  |x|^{\alpha} dx \\   
&=& |S| \int_0^{\infty}  \frac{t}{(t^2+r^2)^{\frac{n+1}{2}} }  r^{\alpha +n-1} dr \\
&\ge& |S| \sum_{k=1}^{\infty}\int_{2^{k-1}t}^{2^kt}   \frac{t}{(t^2+r^2)^{\frac{n+1}{2}} }  r^{\alpha +n-1} dr \\
&\gtrsim & |S|  \sum_{k=1}^{\infty} 2^{k-1}t\frac{t}{(t^2+2^{2k}t^2)^\frac{n+1}{2}}(2^{k-1}t)^{\alpha +n-1}\\
&\gtrsim& t^{\alpha}|S| 2^{-\alpha -n -1}\sum_{k=1}^{\infty}2^{(\alpha -1)k}
\end{eqnarray*}
We see that this sum converges if and only if $\alpha -1<0$. If $n\ge2$ we have $\w_{\alpha}\in A_2$ if and only if $|\alpha|<n$ so we can easily pick a valid $\alpha$ for which the above sum diverges. 

Thus not every weight in $A_2$ is in $\tilde{A}_2$. The converse is still true, though. Let $\w \in \tilde{A}_2$, and let $B$ be a ball with center $a$ and radius $r$. Then for $y \in B$, $|a-y| < r$, and so 
 $$ \frac{1}{r^n} \leq 2^{\frac{(n+1)}{2}}\frac{  r }{(r^2 + |a-y|^2)^{\frac{n+1}{2}}}  $$ and so
$$\langle \w \rangle_B \lesssim  \int_B \frac{  r \,\ \w(y)}{(r^2 +
  |a-y|^2)^{\frac{n+1}{2}}} dy \lesssim P_r(\w)(a),$$ and similarly for
$\langle \w^{-1} \rangle_B.$ Thus $Q_2(\w) \lesssim
\tilde{Q}_2(\w)$. This concludes the proof of Theorem \ref{PoissonA2neqA2}

\section{Remarks on $L^p_{\w}$}\label{section_p}

When defining the appropriate Poisson-$A_p$ class $\tilde{A}_p$ consisting of those weights so that
\begin{equation}\label{apdef_p}
\tilde{Q}_p(\w) := \sup_{(x,t)\in \R^n\times \R_+} P_t(\w)(x)(P_t(\w^{-1/(p-1)})(x))^{p-1} <\infty,
\end{equation} 
our dimension-free estimate holds for $1<p<\infty$: 
\begin{theorem} \label{lbd_p}
There exists a constant $c_p$ that does not depend on the dimension $n$ or on the weight $\w$ so that for all weights $\w \in \tilde{A}_p$ the Riesz vector  as an operator in weighted space $L^p_{\w} \to L^p_{\w}$ has operator norm $\| \vec{R}\|_{L^p_{\w} \to L^p_{\w}} \le c_p \tilde{Q}^{r_p}_p(\w)$ with $r_p=1$ when $p\ge 2$ and $r_p=1/(p-1)$ for $1<p<2$. 
\end{theorem}

We make some remarks about the proof.

In \cite{DGPP} a sharp extrapolation theorem was proven, that can be seen to hold for vector valued operators. In particular, it supplies us with an $L^p$ version of Wittwer's estimate \cite{W} in dimension 1 and more importantly the vector analog Lemma \ref{Wittwervector} in terms of the classical $A_p$ characteristic: 
$$\sup_{\sigma}\|\vec{T}_{\sigma}\|_{L^p_{\w}\to L^p_{\w}}\le c_p Q_p(\w)^{r_p}$$ 
where $r_p$ is as in Theorem \ref{lbd_p}. One derives from this estimate a Bellman function for the $L^p$ case with slightly different variables. The corresponding scalar Bellman function is stated in \cite{PV} p. 299, from which one can deduce the variables in the vector case by replacing $x^p$ by $\| \bx \|^p$ and $y^{p'}$ by $\| \by \|^{p'}$. The remaining part of the argument is identical to the case $p=2$. The resulting estimate is dimensionless and the powers of the respective characteristic of the weight is inherited from the dyadic case. As before, the classical dyadic characteristic that matches a dyadic martingale is replaced by the Poisson characteristic for space time Brownian motion.

\section{Sharpness for $p\neq 2$}\label{section_sharp}

In this section, we show that the estimates claimed in Theorem \ref{lbd_p} are optimal in the sense that for every $1<p<\infty$ there exists no function $\Phi_p:[1,\infty[\to \mathbb{R}_+$ so that $\frac{\Phi_p(x)}{x^{r_p}}\to 0$ when $x\to \infty$ with $\| \vec{R}\|_{L_{\w}^p\to L_{\w}^p}\le c_p\Phi_p(\tilde{Q}_p(\w))$ for all weights $\w\in \tilde{A}_p$ and all $n\in \mathbb{N}$.

One sees that for $n=1$, $| \alpha | < 1$, $\alpha \neq 0$
\begin{eqnarray}
\nonumber \int_{\mathbb{R}} \frac{1}{(1 + | x |^2)} | x |^{\alpha} d x 
\nonumber&=& 2\int_{1}^{\infty}\frac{1}{(1 + x^2)}  x ^{\alpha} d x + 2\int_{0}^{1}\frac{1}{(1 + x^2)}  x ^{\alpha} d x\\
\nonumber&\le& 2\int_{1}^{\infty} x ^{\alpha-2} d x + 2\int_{0}^{1}  x ^{\alpha} d x\\
\nonumber&=& -\frac{2}{\alpha-1}  + \frac{2}{\alpha+1} .
\end{eqnarray}
Similarly $\int_{\mathbb{R}} \frac{1}{(1 + | x |^2)} | x |^{\alpha} d x \ge  -\frac{1}{\alpha-1}  + \frac{1}{\alpha+1} $.
Let $p < 2$ and $0 < s < 1$. Choose Buckley's example $\w_s (x) = | x |^{(p - 1) (1 -
s)}$ and its conjugate weight $\sigma_s(x)=\w_s^{-1/(p-1)}=|x|^{-1+s}$. So the above estimate is valid for these exponents
since $(p - 1) (1 - s) < 1$. Also $-1+s > - 1$ and so one can calculate
\[P_1(\w_s)(0)\le \frac2{2-p+ps-s}+\frac{2}{p-ps+s}\]
\[P_1(\w_s^{-\frac1{p-1}})(0)\le \frac2{2-s} +\frac2{s} .\]
When $p<2$ one has for $s\to 0$ that $P_1(\w_s)(0)\sim 1$ and that $P_1(\w_s^{-1/(p-1)})(0)\lesssim \frac1{s}$. So
\[ P_1 (\w_s) (0) P_1 \left( \w_s^{- \frac{1}{p - 1}} \right)^{p-1} (0) \lesssim 
   \frac{1}{s^{p - 1}} . \]

But for $p<2$ one can see that the same estimate  $$P_t (\w_s) (a) P_t \left( \w_s^{-
\frac{1}{p - 1}} \right)^{p-1} (a)\lesssim \frac1{s^{p-1}}$$ holds when $t \neq 1$ and $| a | > 0$. Indeed, when $t \neq 1$
one adjusts the calculation to the scale $t$ and the above product is
independent of $t$. One then shows that the quantity is decreasing in $| a |$.
So
\[ \tilde{Q}_p (| x |^{(p - 1) (1 - s)}) \lesssim s^{1 - p} . \]
Now take $f_s (x) = x^{s - 1} \chi_{]0, 1]}$. One estimates for $x > 2$ thanks to the support of $f_s$ and the
decay of the kernel of $H$ that
\[ H f_s (x) \simeq \int_0^1\frac1{x-t}f_s(t)dt \gtrsim \frac1{x}\int_0^1 f_s(t)dt = \frac{1}{s x} . \]
Then
\[ \| H f_s \|^p_{L_{\w_s}^p} \gtrsim \frac{1}{s^p} \int^{\infty}_2 x^{(p - 1) (1
   - s) - p} d x \gtrsim \frac{- 1}{s^p (1 - p) s} \gtrsim \frac1{s^{p+1}}\]
and
\[ \| f_s \|_{L_{\w_s}^p}^p = \int^1_0 x^{p (s - 1)} x^{(p - 1) (1 - s)} d x =
   \int^1_0 x^{s - 1} d x = \frac{1}{s} . \]
So
\[ \tilde{Q}_p (\w_s)^{\frac{1}{p - 1}} \| f_s \|_{L_{\w_s}^p} \lesssim s^{- 1 -
   \frac{1}{p}} \tmop{and} \| H f_s \|_{L_{\w_s}^p} \gtrsim s^{- 1 - \frac{1}{p}}
   . \]
   
Letting $s \rightarrow 0$ shows that the estimate is optimal for $1<p< 2$. The range
of $p > 2$ is seen by duality $(L_{\w}^p)^{\ast} = L_{\w^{1 - p'}}^{p'} $
and the fact that $H$ is self adjoint up to a sign. We detail the argument: let $\w_s = | x |^{s -
1}$ then $P_t (\w_s)$ and $P_t \left( \w_s^{\frac{- 1}{p - 1}} \right)$ converge
for $p > 2$ and $\tilde{Q}_{p'} (\w_s^{1 - p'}) = \tilde{Q}_p (\w_s)^{\frac{1}{p
- 1}}$. Assume there exists $\Phi$ growing slower than linear so that for some $p > 2$ we
have $\| H \|_{L_{\w_s}^p} \lesssim \Phi (\tilde{Q}_p (\w_s))$. Then $\| H^{\ast}
\|_{(L_{\w_s}^p)^{\ast}} \lesssim \Phi (\tilde{Q}_p (\w_s))$ and thus with $v_s=\w_s^{1-p'}$ we have $\| H
\|_{L_{v_s}^{p'} } \lesssim \Phi (\tilde{Q}_{p'} (v_s)^{p - 1})$
with $p' < 2 $. This contradicts the sharpness already seen for this
range of exponents with weights $v_s \in \tilde{A}_{p'}$.

One can see that for $p=2$ we have $P_1(\w_s)(0)\gtrsim \frac1{s}$, indeed $P_1(\w_s)(0)\ge \frac1{2-p +ps-s}+\frac{1}{p-ps +s}$ and that 
the example provided above does not give sharpness at the critical exponent $p=2$. We get to this exponent through extrapolation in a martingale setting.

\section{Martingale Extrapolation}\label{section_extrapolation}

\

In the following we have a filtered probability space with the usual
assumptions: the filtration is right continuous and $\mathcal{F}_0$ contains
all $\mathcal{F}$ null sets. Let $w$ be a positive, uniformly integrable
martingale, called a weight. Let $X$ and $Y$ be adapted right continuous
martingales. \ Throughout, we may identify martingales with their closures,
for example $w_{\infty}$ with $w$ by the assumption of uniform integrability
of $w$. The $A^{\mathcal{F}}_p$ characteristic in this setting is
\begin{eqnarray*}
  Q^{\mathcal{F}}_p (w) = [w]_{A^{\mathcal{F}}_p} = \sup_{\tau} \| w_{\tau}
  \sigma^{p - 1}_{\tau} \|_{\infty}
\end{eqnarray*}
where $\tau$ adapted stopping times and $\sigma$ the conjugate weight so that
$\sigma^p w = \sigma$. The case that will interest us is just two dimensional
Brownian motion (or rather: background noise) with its induced
filtration.

The theorem below is referred to as extrapolation theorem and appeared in its sharp form in \cite{DGPP}
for a pair of functions $f$ and $T f$ with $T$ a sublinear operator and in the
case of classical weight characteristics
\begin{eqnarray*}
  \sup_Q \frac{1}{| Q |} \int_Q w \frac{1}{| Q |} \int_Q w^{- 1}
\end{eqnarray*}
with the supremum over all cubes in Euclidean space without underlying
filtrations. Our statement below involves martingales and filtered probability
spaces - we take special care of the quantifiers that appear since we plan to
use extrapolation for a lower estimate to recover sharpness in the critical
exponent $p = 2$.

\begin{theorem}
  Given a filtered probability space as described above. Let $1 < p < \infty$
  and $w \in A^{\mathcal{F}}_p$. Let martingales $X, Y \in L^p_w$. Suppose
  $1 < r < \infty$ and suppose $\forall A \geqslant 1, \exists N_r (A) > 0$
  increasing such that for triples $X, Y, \rho$ with $Y, X \in L^r_{\rho}$ and
  $Q_r^{\mathcal{F}} (\rho) = [\rho]_{A^{\mathcal{F}}_r} \leqslant A$
  \begin{eqnarray*}
    \| Y \|_{L^r_{\rho}} \leqslant N_r (A) \| X \|_{L^r_{\rho}} .
  \end{eqnarray*}
  Then for any $1 < p < \infty$ there exists $N_p (B) > 0$ such that if
  $Q^{\mathcal{F}}_p (w) = [w]_{A^{\mathcal{F}}_p} \leqslant B$ there holds
  \begin{eqnarray*}
    \| Y \|_{L^p_w} \leqslant N_p (B) \| X \|_{L^p_w} .
  \end{eqnarray*}
  With $C^{\ast} (p)$ denoting the numeric part of the estimate in the
  weighted $L^p$ maximal estimate, in particular if $p > r$ then $N_p (B)
  \leqslant 2^{1 / r} N_r (2 C^{\ast} (p')^{(p - r) / (p - 1)} B)$. If $p < r$
  then $N_p (B) \leqslant 2^{(r - 1) / r} N_r \left( 2^{r - 1} (C^{\ast}
  (p)^{p - r} B)^{\frac{r - 1}{p - 1}} \right)$.
\end{theorem}

\begin{remark}
The extrapolation theorem also holds of course for filtered spaces with discrete time.
\end{remark}

In the classical setting, Buckley's sharp estimate for the weighted
Hardy-Littlewood maximal function
\begin{eqnarray*}
  \| M \|_{L^p_w \rightarrow L^p_w} \lesssim Q_p (w)^{\frac{p'}{p}}
\end{eqnarray*}
appeared in the proof of the extrapolation theorem. Here, it is replaced by a
weighted Doob inequality:
\begin{eqnarray*}
  \| X^{\ast} \|_{L^p_w} \lesssim [w]^{\frac{p'}{p}}_{A^{\mathcal{F}}_p} \|
  X \|_{L^p_w} .
\end{eqnarray*}
For this estimate, see for example \cite{O} if one is satisfied with
continuous in space processes. Consider \cite{DP3} if jumps are desired. The latter
also gives an estimate of the implied constant $C^{\ast} (p) = \frac{p^{p'}}{p
- 1}$.

The rest of the modifications are minor for readers acquainted with basic
probability theory. We will pass back and forth between martingales, say $M_t$
and their closures $M_{\infty}$. This is possible due to the assumption of
uniform integrability on the weight $w$ or it is clear from the context that
we are dealing with a right continuous submartingale with $\sup_t \mathbb{E}
(| M_t |) < \infty$ so that the martingale convergence theorem applies. 
Sometimes apply the convergence theorem to the submartingale $N_t = | M_t |^s$ and thus
obtain a closure $N_{\infty}$ and $M_{\infty}$. To compute martingale $L^p$
norms, we may now work with their closures. Since the integrability
assumptions on the martingales $X, Y$ will be essential to our argument of our
lower estimate, we sketch the proof of the extrapolation theorem briefly. Just as in
\cite{DGPP} we need the following lemma:

\begin{lemma}
  \label{lemmaV}Let $1 < p < \infty$ and let $w \in A^{\mathcal{F}}_p$.
  \begin{enumerate}
    \item Let $1 < r < p < \infty$ and pose $s = (p / r)' = p / (p - r)$. Then
    for every non-negative $u \in L^s_w$ there exists $v \in L^s_w$ such
    that almost everywhere $u (x) \leqslant v (x)$ and $\| v \|_{L^s_{w}}
    \leqslant 2 \| u \|_{L^s_w}$ and with martingale $v w = (v w)_t$ we have
    $[v w]_{A^{\mathcal{F}}_r} \lesssim [w]_{A^{\mathcal{F}}_p}$ with the
    implied constant depending only on $p, r$.
    
    \item Let $1 < p < r$ and pose $s = p / (r - p)$. Then for every
    non-negative $u \in L^s_w$ there exists $v \in L^s_w$ such that almost
    everywhere $u (x) \leqslant v (x)$ and $\| v \|_{L^s_w} \leqslant 2^{r -
    1} \| u \|_{L^s_w}$ and $[v^{- 1} w]_{A^{\mathcal{F}}_r} \lesssim
    [w]^{\frac{r - 1}{p - 1}}_{A^{\mathcal{F}}_p}$ with the implied constant
    depending only on $p, r$. 
  \end{enumerate}
\end{lemma}

\begin{proof}
  Let us consider $S (u)^s = (w^{- 1} (u^{s / p'} w)^{\ast})^{p'}$. Observe in
  a straightforward manner, using the maximal estimate
  \begin{eqnarray}
    \| S (u) \|_{L^s_w} \leqslant C^{\ast} (p')^{p' / s} [w^{1 - p'}]^{p /
    p' \cdot p' / s}_{A^{\mathcal{F}}_{p'}} \| u \|_{L^s_w} = [w]^{p' /
    s}_{A^{\mathcal{F}}_p} \| u \|_{L^s_w} \label{inequalitySnorm}
  \end{eqnarray}
  where we used that $[w^{1 - p'}]_{A^{\mathcal{F}}_{p'}} = [w]^{p' /
  p}_{A^{\mathcal{F}}_p}$. We also show that $(u w, S (u) w)$ belongs to
  $A^{\mathcal{F}}_r$ with characteristic bounded by $[w]^{1 - p' /
  s}_{A^{\mathcal{F}}_p}$. The last statement concerns martingales $(u w)_t$
  and $((S (u) w)^{- 1 / (r - 1)})_t$. Let $\tau$ be a stopping time. There
  holds
  \begin{eqnarray*}
    &&(u w)_{\tau} ((S (u) w)^{- 1 / (r - 1)})^{r - 1}_{\tau} \\
    & = & (u w^{p' /
    s} w^{1 - p' / s})_{\tau} (((u^{s / p'} w)^{\ast})^{- p' / s (r - 1)} w^{-
    1 / (p - 1)})^{r - 1}_{\tau}\\
    & \leqslant & (u^{s / p'} w)^{p' / s}_{\tau} (w)^{1 - p' / s}_{\tau}
    (u^{s / p'} w)^{- p' / s}_{\tau} (w^{- 1 / (p - 1)})^{r - 1}_{\tau}\\
    & = & ((w)_{\tau} (w^{- 1 / (p - 1)})^{p - 1}_{\tau})^{1 - p' / s} .
  \end{eqnarray*}
  In the second line we used H{\"o}lder inequality for the first term and for
  the second term that for any stopping time $\tau$ almost everywhere $(u^{s /
  p'} w)_{\tau} \leqslant (u^{s / p'} w)^{\ast}$ while observing that the
  exponent $- p' / s (r - 1) < 0$, in combination with elementary property of
  conditional expectation. The last line uses $r - 1 = (p - 1) (1 - p' / s)$.
  Taking supremum over all stopping times gives the estimate
  \begin{eqnarray}
    [u v , S (u) w \nocomma]_{A^{\mathcal{F}}_r} \leqslant [w]^{1 - p'
    / s}_{A^{\mathcal{F}}_p} . \label{inequalityAr}
  \end{eqnarray}
  For part a) we let $v = \sum^{\infty}_{n = 0} \frac{S^n (u)}{2^n \| S
  \|^n}$. Observe $S (v) \leqslant 2 \| S \| (v - u) \leqslant 2 \| S \| v$.
  Use this observation and the $A^{\mathcal{F}}_r$ estimate
  (\ref{inequalityAr}) and the above norm estimate (\ref{inequalitySnorm}) of
  $\| S \|$ to obtain
  \begin{eqnarray*}
    (v w)_{\tau} ((v w)^{- 1 / (r - 1)})^{r - 1}_{\tau} & \leqslant & 2 \| S
    \| (v w)_{\tau} ((S (v) w)^{- 1 / (r - 1)})^{r - 1}_{\tau}\\
    & \leqslant & 2 C^{\ast} (p')^{p' / s} [w]_{A^{\mathcal{F}}_p}.
  \end{eqnarray*}
  For part b) use duality, see \cite{DGPP} for details.
\end{proof}

We pass to the proof of the extrapolation theorem for martingales

\begin{proof}
  Assume first $1 < r < p$. Identify $Y$ with its closure.
  \begin{eqnarray*}
    \| Y \|_{L^p_w}^r = \| | Y |^r \|_{L^{s'} (w)} = \sup_{u \geqslant 0 :
    \| u \|_{L^s_w} = 1} \int | Y |^r u w d\mathbb{P}.
  \end{eqnarray*}
  Take $v$ as constructed in Lemma \ref{lemmaV} and obtain
  \begin{eqnarray*}
    \int | Y |^r u w d\mathbb{P} & \leqslant & \int | Y |^r v w
    d\mathbb{P}\\
    & \leqslant & N_r ([v w]_{A^{\mathcal{F}}_r})^r \int | X |^r v w^{r / p}
    w^{1 - r / p} d\mathbb{P}\\
    & \leqslant & N_r ([v w]_{A^{\mathcal{F}}_r})^r \left( \int | X |^p w
    d\mathbb{P} \right)^{r / p} \left( \int v^s w d\mathbb{P} \right)^{1 /
    s}\\
    & \leqslant & 2 N_r ([v w]_{A^{\mathcal{F}}_r})^r \| X \|^r_{L^p_w}.
  \end{eqnarray*}
  Thanks to the estimate on $[v w]_{A^{\mathcal{F}}_r}$ and $N_r$ increasing
  we have the desired estimate after taking supremum in admissible $u$. Note
  that in particular $X, Y$ belong to $L^r_{v w}$ for all admissible $u$ and
  their so constructed $v$. The case $1 < p < r$ is similar, see \cite{DGPP} with
  similar changes as above to the setting here. 
\end{proof}

\

\section{The probabilistic Hilbert transform and $A_p$--characteristic}

\

Let $1 < p < 2$. Choose $f_s$ and $\w_s$ as above. We have seen that $\| H f_s
\|_{L^p_{\w_s}} \gtrsim s^{- 1 - 1 / p}$ and that $\| f_s \|_{L^p_{\w_s}}
\lesssim s^{- 1 / p}$ and $Q_p (\w_s)^{1 / (p - 1)} \lesssim s^{- 1}$.

\

We wish to recast those estimates in a probabilistic setting for suitable
martingales, so as to be able to use the extrapolation result stated above.
Recall the probabilistic interpretation of the Hilbert transform following
Gundy--Varopoulos {\cite{GV}}.

Let $f (x)$ defined on $\mathbb{R}$. Let $\tilde{f} (x, y)$ its harmonic
extension in the upper--half space $\mathbb{R}^2_+$. Let $W_t \assign (x_t,
y_t)$, $t \leqslant 0$, the so-called background noise built in
{\cite{GV}}. These paths are based on $2$--dimensional Brownian
motion, starting at infinity at time $t = - \infty$ and hitting the boundary
of the upper--half space at time $t = 0$, i.e. $W_0 = (x_0, 0)$ for some
random $x_0 \in \mathbb{R}$. Then $M_t^{\tilde{f}} \assign \tilde{f} (W_t)$
is a martingale and It{\^o} formula writes for all $t \leqslant 0$,
\begin{eqnarray*}
  M_t^{\tilde{f}} \assign \widetilde{f} (W_t) & = & \int_{- \infty}^t \nabla
  \tilde{f} (W_{s -}) \cdot \mathd W_s .
\end{eqnarray*}
We have similarly, setting $g \assign H f$, that $M_t^{\tilde{g}} \assign
\tilde{g} (W_t)$ is a martingale, and for all $t \leqslant 0$,
\begin{eqnarray*}
  \tilde{g} (W_t) & = & \int_{- \infty}^t \nabla \tilde{g} (W_{s -}) \cdot
  \mathd W_s\\
  & = & \int_{- \infty}^t \nabla^{\perp} \tilde{f} (W_{s -}) \cdot \mathd
  W_s,
\end{eqnarray*}
where $\nabla \tilde{g} = \nabla^{\perp} \tilde{f}$ are the Cauchy--Riemann
relations, with $\nabla^{\perp} \assign (- \partial_y, \partial_x)$. Notice
that in the Hilbert transform case, conditioning by arrival point is not
needed (as it is in the case of Riesz transform). Therefore the probabilistic
interpretation of the Hilbert transform does not involve a projection
operator.

\subsection{$A_p$ characteristics in $\mathbb{R}_+^{n + 1}$}

Let $w > 0$, $\sigma = w^{- 1 / (p - 1)} > 0$, and $\tilde{w}$,
$\tilde{\sigma}$ their harmonic extensions. Set
\begin{eqnarray*}
  \tilde{Q}_p (w) \assign \sup_{x, y \in \mathbb{R}_+^2} \tilde{w} (x, y) 
  \tilde{\sigma} (x, y)^{(p - 1)} .
\end{eqnarray*}
Introduce the two martingales $\tilde{w}_t \assign \tilde{w} (W_t)$ and
$\tilde{\sigma}_t \assign \tilde{\sigma} (W_t)$ and set
\begin{eqnarray*}
  Q_p^{\mathcal{F}} (w) \assign \sup_{\tau} \tmop{ess} \sup_{\omega} 
  \tilde{w}_{\tau}  \tilde{\sigma}_{\tau}^{(p - 1)}
\end{eqnarray*}
where the supremum is over all adapted stopping times $\tau$ and an
$L^{\infty}$ norm arises in $\omega \in \Omega$. We want to prove that
$\tilde{Q}_p (w) = Q^{\mathcal{F}}_p (w)$. From the definitions it is clear
that $\tilde{Q}_p (w) \geqslant Q^{\mathcal{F}}_p (w)$. Rewrite now
\begin{eqnarray*}
  \tilde{Q}_p (w) \assign \sup_{y \in \mathbb{R}_+}  \{ \sup_{x \in
  \mathbb{R}} \tilde{w} (x, y)  \tilde{\sigma} (x, y)^{(p - 1)} \} .
\end{eqnarray*}
For any given $y \geqslant 0$, setting $\tau_y = \inf \{ s : y_s = y \}$, the
translation invariance of the background noise ensures that $W_{\tau_y} =
(x_{\tau_y}, y)$ is a random variable with density proportionally to the
Lebesgue measure on the line $\mathbb{R} \times \{ y \}$, hence \
\begin{eqnarray*}
  \sup_{x \in \mathbb{R}} \tilde{w} (x, y)  \tilde{\sigma} (x, y)^{(p - 1)}
  \leqslant \tmop{ess} \sup_{\omega}  \tilde{w}_{\tau_y} 
  \tilde{\sigma}_{\tau_y}^{(p - 1)} .
\end{eqnarray*}
Letting $y$ span $\mathbb{R}_+$ yields the result.

\

\section{Sharpness at $p = 2$}

\

We proceed by contradiction. Let us assume that
\begin{eqnarray}
  \| H f \|_{L^2_w} \leqslant N_2 (\tilde{Q}_2 (w)) \| f \|_{L^2_w}
  \label{eq: HfL2 sublinear w.r.t f}
\end{eqnarray}
for all weights in $\tilde{A}_2$ and all admissible $f$ with a function $N_2$
growing sublinear at infinity. Choose $p$ with $1 < p < 2$. Choose $s$, with
$0 < s < 1$ and introduce as before the Buckley weight $\w_s = | x |^{(p - 1)
(1 - s)}$, the function $f_s (x) \assign x^{s - 1} \chi_{[0, 1]}$ and its
Hilbert transform $g_s = H f_s$. We have seen that $\w_s \in A_p$, $H f_s
\tmop{and} f_s \in L^p_{\w_s}$.

Introduce now the corresponding martingales $X^{(s)} \assign
M_t^{\widetilde{f_s}}$, $Y^{(s)} \assign M_t^{\widetilde{g_s}}$,
$(\tilde{\w}_s)_t$. Observe $\| X^{(s)} \|_{L^p_{\w_s}} = \| f_s \|_{L^p_{\w_s}}$, $\| Y^{(s)} \|_{L^p_{\w_s}} = \| H f_s \|_{L^p_{\w_s}}$ and
$\tilde{Q}_p (\w_s) = Q_p^{\mathcal{F}} (\w_s)$.

Therefore extrapolation (see also the remark at the end of the proof) provides
us with a sequence of $A^{\mathcal{F}}_2$ weights $\rho_s$ where $X^{(s)}$ and
$Y^{(s)}$ belong to $L^2_{\rho_s}$ that will allow us to extrapolate to $p$.
Thanks to the martingale equivalent of (\ref{eq: HfL2 sublinear w.r.t f}) we
have in particular $\| Y^{(s)} \|_{L^2_{\rho_s}} \leqslant N_2
(Q^{\mathcal{F}}_2 (\rho_s)) \| X^{(s)} \|_{L^2_{\rho_s}}$. Extrapolating to
$p$ we obtain
\begin{eqnarray*}
  \| Y^{(s)} \|_{L^p_{\w_s}} \leqslant 2^{1 / 2} N_2 (2 (C^{\ast} (p)^{p - 2}
  Q^{\mathcal{F}}_p (\w_s))^{1 / (p - 1)}) \| X^{(s)} \|_{L^p_{\w_s}} .
\end{eqnarray*}
Coming back to the deterministic setting, this is exactly
\begin{eqnarray}
  \| H f_s \|_{L^p_{\w_s}} \lesssim N_2 ((\tilde{Q}_p (\w_s))^{1 / (p - 1)}) \|
  f_s \|_{L^p_{\w_s}} . \label{eq: HfLp sublinear w.r.t f}
\end{eqnarray}
However for $p < 2$, we have obtained in section \ref{section_sharp} the quantitative estimates 
$\tilde{Q}_p (\w_s)^{1 / (p - 1)} \lesssim s^{- 1}$, $\| f_s \|_{L^p_{\w_s}}
\lesssim s^{- 1 / p}$ and $\| H f_s \|_{L^p_{\w_s}} \gtrsim s^{- 1 - 1 / p}$.
When $s \rightarrow 0$, this contradicts the estimate (\ref{eq: HfLp sublinear
w.r.t f}) above since we assumed sublinearity for $N_2$.

\end{document}